\documentclass[twoside,12pt]{article}
\usepackage{times,amsmath,amssymb,amsthm,cite}
\newcommand{\subj}[1]{\par\noindent{\bf AMS Subject Classifications: }#1.}
\newcommand{\keyw}[1]{\par\noindent{\bf Keywords: }#1.}
\numberwithin{equation}{section}
\numberwithin{figure}{section}
\newtheorem{theorem}{Theorem}[section]

\newtheorem{corollary}[theorem]{Corollary}

\theoremstyle{definition}
\newtheorem{definition}[theorem]{Definition}
\newtheorem{example}[theorem]{Example}
\theoremstyle{remark}
\newtheorem{remark}[theorem]{Remark}

\everymath{\displaystyle}
\date{}
\newcommand{\adsa}
{\vspace{-1in}\normalsize\flushleft
This is a preprint of a paper whose final and definite form will appear in \\
Advances in Dynamical Systems and Applications, ISSN 0973-5321\\
{\tt http://campus.mst.edu/adsa}\\\vspace{1mm}\hrule\vspace{5mm}
\renewcommand\thefootnote{{}}
\footnotetext{\noindent\tt Received February 11, 2017; Accepted June 22, 2017\\
\hspace*{8pt}Communicated by Delfim F. M. Torres}}
\begin{document}
\title{\adsa\center\Large\bf Time Scale Version of the Hermite--Hadamard Inequality for Functions Convex on the Coordinates }
\author{{\bf Eze R. Nwaeze}\\
        Tuskegee University\\
        Department of Mathematics\\
        Tuskegee, 36088, USA\\
        {\tt enwaeze@mytu.tuskegee.edu}
}
\maketitle
\thispagestyle{empty}

\begin{abstract}
In this paper, we present a time scale version of the Hermite-Hadamard inequality for functions convex on the coordinates via the diamond-$\alpha$ calculus. Our results are new and they generalize and extend a result due to Dragomir.
\end{abstract}
\subj{26D15, 54C30, 26D10}
\keyw{Hermite-Hadamard inequality, diamond-$\alpha$, convex function, time scales}

\section{Introduction}
A function $f:I\rightarrow\mathbb{R},~\emptyset\neq I\subseteq\mathbb{R},$ is said to be convex on the interval $I$ if the inequality $$f(\lambda x + (1-\lambda)y)\leq \lambda f(x)+(1-\lambda)f(y),$$
holds for all $x, y\in I$ and $\lambda\in [0, 1].$
A well celebrated inequality for the class of convex functions is the Hermite-Hadamard's inequality. The inequality states that for any convex function $f:[a, b]\rightarrow\mathbb{R}$ we have

\begin{equation}\label{Had}
f\Big(\frac{a+b}{2}\Big)\leq \frac{1}{b-a}\int_a^b f(x)dx\leq \frac{f(a)+f(b)}{2}.
\end{equation}

Using (\ref{Had}), Dragomir \cite{Drag} proved the following Hadamard's type result for functions, defined on a rectangle, that are convex on the coordinates.
\begin{theorem}\label{drag}
Let $f:[a, b]\times [c, d]\rightarrow\mathbb{R}$ be such that the partial mappings
\begin{equation*}
f_{y}:[a,b]\rightarrow \mathbb{R},~f_{y}(u):=f(u,y)~and~f_{x}:[c,d]
\rightarrow \mathbb{R},~f_{x}(v):=f(x,v)
\end{equation*}
defined for all $y\in [c,d]$ and $x\in [a,b],$ are convex. Then,
\begin{align*}
&f\Big(\frac{a+b}{2},\frac{c+d}{2}\Big )\\
&\leq\frac{1}{2}\Bigg[\frac{1}{b-a}\int_a^bf\Big(x,\frac{c+d}{2}\Big) dx+\frac{1}{d-c}\int_c^df\Big(\frac{a+b}{2},y\Big) dy \Bigg]\\
\nonumber
 &\leq \frac{1}{(b-a)(d-c)}\int_a^b\int_c^df(x, y) ~dx dy\\
\nonumber
&\leq \frac{1}{4(b-a)}\int_a^b\Big[f(x,c) +f(x,d)\Big] dx+\frac{1}{4(d-c)}\int_c^d\Big[f(a, y)+f(b, y)\Big] dy\\
&\leq \frac{f(a, c) + f(a, d) +f(b, c) + f(b, d)}{4}.
\end{align*}
The above inequalities are sharp.
\end{theorem}
In order to unify the theory of integral and differential calculus with the calculus of finite difference, Hilger \cite{Hilger} in 1988 introduced the concept of time scales (see Section \ref{sec:Prelim} for a brief overview). Since the advent of this notion, many classical mathematical inequalities have been extended to time scales. In this subject, the concept of the delta and nabla differentiation (integration) are introduced. In 2006, Sheng et al.\cite{Sheng} considered the linear combination of these already established derivatives (integrals) on time scales. This they called the diamond-$\alpha$ derivative (integrals) (see Definitions \ref{diamD} and \ref{diamI}). This new combined dynamic calculus has since generated a lot of interest among mathematicians, see for example, \cite{Ammi,Dinu,Tuna3} and the references therein.

In 2008, Dinu \cite{Dinu} extended (\ref{Had}) to time scales by proving the following result.

\begin{theorem}\label{lem1}
Let $\mathbb{T}$ be a time scale and $a, b\in \mathbb{T}.$ Let $f:[a, b]\cap\mathbb{T}\rightarrow\mathbb{R}$ be a continuous convex function. Then, $$f(t_{\alpha})\leq \frac{1}{b-a}\int_a^b f(x)\Diamond_{\alpha}x\leq \frac{b-t_{\alpha}}{b-a}f(a) + \frac{t_{\alpha}-a}{b-a}f(b),$$
where $t_\alpha:=\frac{1}{b-a}\int_a^b t \Diamond_{\alpha}t.$
\end{theorem}

The purpose of this paper is to extend Theorem \ref{drag} to time scales via the combined diamond-$\alpha$ dynamics. We do this by using Theorem \ref{lem1}. Some other interesting results are also obtained  in this direction.

The paper is organized as follows: in Section \ref{sec:Prelim}, we give a brief background of the theory of time scales. Thereafter, our results are formulated and proved in Section \ref{sec:MR}.


\section{Preliminaries}\label{sec:Prelim}
~~A {\it time scale} $\mathbb{T}$ is an arbitrary nonempty closed subset of $\mathbb{R}.$  We assume throughtout that a time scale $\mathbb{T}$ has the topology that it inherits from the real numbers with the standard topology. In what follows, we will give a brief overview of some basic notions connected to time scales and differentiability of functions on them.

The forward {\it jump operator} $\sigma:\mathbb{T}\rightarrow\mathbb{T}$ is defined by $$\sigma(t):=\inf\{s\in\mathbb{T}: s>t\},$$
while the backward {\it jump operator} $\rho:\mathbb{T}\rightarrow\mathbb{T}$ is defined by $$\rho(t):=\sup\{s\in\mathbb{T}: s<t\}.$$
In this definition, we put $\inf \emptyset=\sup\mathbb{T}$ (i.e., $\sigma(t)=t$ if $\mathbb{T}$ has a maximum $t$) and $\sup \emptyset=\inf\mathbb{T}$ (i.e., $\rho(t)=t$ if $\mathbb{T}$ has a minimum $t$), where $\emptyset$ denotes the empty set.
If $\sigma(t)>t,$ we say that $t$ is {\it right-scattered}, while if $\rho(t)<t$ we say that $t$ is {\it left-scattered}. Points that are right-scattered and left-scattered at the same time are called {\it isolated}. Also, if $t < \sup\mathbb{T}$ and $\sigma(t) = t,$ then $t$ is called right-dense, and if $t > \inf \mathbb{T}$ and $\rho(t) = t,$ then $t$ is called left-dense. Points that are right-dense
and left-dense at the same time are called {\it dense}. We also introduce the sets $\mathbb{T}^k$, $\mathbb{T}_k$, and $\mathbb{T}^k_k$, which are derived from the time scale $\mathbb{T}$ as follows:
if $\mathbb{T}$ has a left-scattered maximum $t_1$, then $\mathbb{T}^k=\mathbb{T}\setminus\{t_1\},$ otherwise $\mathbb{T}^k=\mathbb{T}.$  If $\mathbb{T}$ has a right-scattered minimum $t_2,$ then $\mathbb{T}_k=\mathbb{T}\setminus\{t_2\},$ otherwise $\mathbb{T}_k=\mathbb{T}.$  Finally, we define $\mathbb{T}^k_k=\mathbb{T}^k\cap\mathbb{T}_k.$

 For $a, b \in \mathbb{T}$ with $a\leq b,$ we define the interval $[a, b]$ in $\mathbb{T}$ by $[a, b]=\{t\in\mathbb{T}: a\leq t\leq b\}.$ Open intervals and half-open intervals are defined in the same manner.

\begin{definition}[Delta derivative]
Assume $f:\mathbb{T\rightarrow R}$ is a function.
Then the delta derivative $f^{\Delta }(t)\in \mathbb{R}$ at $t\in \mathbb{T%
}^{k}$ is defined to be number (provided it exists) with property that given
for any $\epsilon >0$ there exists a neighborhood $U$ of $t$ such that%
{\small
\begin{equation*}
\left\vert f(\sigma (t))-f(s)-f^{\Delta }(t)\left[ \sigma (t)-s\right]
\right\vert \leq \epsilon \left\vert \sigma (t)-s\right\vert ,\text{ \ \ \ }%
\forall s\in U.
\end{equation*}%
}
\end{definition}

\begin{definition}[Nabla derivative]
Assume $f:\mathbb{T\rightarrow R}$ is a function.
Then the nabla derivative $f^{\nabla }(t)\in \mathbb{R}$ at $t\in \mathbb{T%
}_{k}$ is defined to be number (provided it exists) with property that given
for any $\epsilon >0$ there exists a neighborhood $V$ of $t$ such that%
{\small
\begin{equation*}
\left\vert f(\rho (t))-f(s)-f^{\nabla }(t)\left[ \rho(t)-s\right]
\right\vert \leq \epsilon \left\vert \rho (t)-s\right\vert ,\text{ \ \ \ }%
\forall s\in V.
\end{equation*}%
}
\end{definition}

\begin{definition}[Delta integral]
A function $F:\mathbb{T}\rightarrow\mathbb{R}$ is said to be a delta antiderivative of  $f:\mathbb{T}\rightarrow\mathbb{R},$ provided $F^{\Delta}(t)=f(t)$ for all $t\in{\mathbb{T}}^k.$ For all $a, b\in\mathbb{T},~a<b,$ the delta integral of $f$ from $a$ to $b$ is defined by
$$\int_a^b f(t)\Delta t = F(b) - F(a).$$
\end{definition}

\begin{definition}[Nabla integral]
A function $G:\mathbb{T}\rightarrow\mathbb{R}$ is said to be a nabla antiderivative of  $g:\mathbb{T}\rightarrow\mathbb{R},$ provided $G^{\nabla}(t)=g(t)$ for all $t\in{\mathbb{T}}_k.$ For all $a, b\in\mathbb{T},~a<b,$ the nabla integral of $g$ from $a$ to $b$ is defined by
$$\int_a^b g(t)\nabla t = G(b) - G(a).$$
\end{definition}

For an indepth study of the properties of the $\Delta$ and $\nabla$ derivatives and integrals, we refer the interested reader to the books \cite{BookTS:2001,BookTS:2003}.\\

Now, we give a brief introduction of the diamond-$\alpha$ dynamic derivative and of the diamond-$\alpha$ integral.

\begin{definition}[\protect\cite{Rogers,Sheng}]
Let $t, s\in\mathbb{T}$ and define $\mu_{ts}:=\sigma(t)-s$ and $\eta_{ts}:=\rho(t)-s.$ We say that a function $f:\mathbb{T}\rightarrow\mathbb{R}$ is diamond-$\alpha$ differentiable at $t\in{\mathbb{T}}^k_k$ if there exists a number $f^{\Diamond_\alpha}(t)$ such that, for all $\epsilon>0,$ there exists a neighborhood $U$ of $t$ such that , for all $s\in U,$
$$\left|\alpha[f(\sigma(t))-f(s)]\eta_{ts} + (1-\alpha)[f(\rho(t))-f(s)]\mu_{ts} - f^{\Diamond_\alpha}(t)\mu_{ts}\eta_{ts}\right|\leq \epsilon|\mu_{ts}\eta_{ts}|.$$
A function $f$ is said to be diamond-$\alpha$ differentiable provided $f^{\Diamond_\alpha}(t)$ exists for all $t\in {\mathbb{T}}^k_k.$
\end{definition}

\begin{theorem}[\protect\cite{Rogers,Sheng}]\label{diamD}
 Let $\mathbb{T}$ be a time scale and $f$ be differentiable on $\mathbb{T}$ in the $\Delta$ and $\nabla$ senses at $t\in{\mathbb{T}}^k_k$ .
Then $f$ is  diamond-$\alpha$ differentiable at $t$ and
$$f^{\Diamond_{\alpha}}(t)=\alpha f^{\Delta}(t) + (1-\alpha) f^{\nabla}(t), ~~0\leq\alpha\leq 1.$$
Thus $f$ is diamond-$\alpha$ differentiable if and only if $f$ is $\Delta$ and $\nabla$ differentiable.
\end{theorem}

We may notice that the diamond-$\alpha$ derivative reduces to the standard $\Delta$ derivative as $\alpha=1,$
or the standard $\nabla$ derivative as $\alpha=0,$ while it represents a ``weighted dynamic derivative'' for
 $\alpha\in (0, 1).$  Furthermore, the combined dynamic derivative offers a centralized derivative
formula on any uniformly discrete time scale $\mathbb{T}$ when $\alpha=1/2.$  Needless to say, the latter
feature is particularly useful in many computational applications.

\begin{definition}[\protect\cite{Sheng}]\label{diamI}
Let $a, t\in\mathbb{T}$ and $f:\mathbb{T}\rightarrow\mathbb{R}.$ We define the $\Diamond_{\alpha}$ integral of $f$ as
$$\int_a^t f(s)\Diamond_{\alpha}s=\alpha\int_a^t f(s)\Delta s + (1-\alpha)\int_a^t f(s)\nabla s, ~~s\in \mathbb{T}, ~~0\leq\alpha\leq 1,$$
provided that there exist delta and nabla integrals of $f$ on $\mathbb{T}.$
\end{definition}
Next, we present some properties of the diamond-$\alpha$ integral that will come handy in the proof of our main results.

\begin{theorem}[\protect\cite{Ammi,Sheng}]\label{use}
Let $f$ and $g$ be two continuous functions on $[a, b],$ $a, b,t\in \mathbb{T},$ and $ c\in\mathbb{R}.$ Then
\begin{enumerate}
\item $\int_a^t [f(s) + g(s)]\Diamond_{\alpha}s =\int_a^t f(s)\Diamond_{\alpha}s + \int_a^t g(s)\Diamond_{\alpha}s,$
\item $\int_a^t cf(s)\Diamond_{\alpha}s = c\int_a^t f(s)\Diamond_{\alpha}s,$
\item $\int_a^t f(s)\Diamond_{\alpha}s=-\int_t^a f(s)\Diamond_{\alpha}s,$
\item $\int_a^t f(s)\Diamond_{\alpha}s=\int_a^b f(s)\Diamond_{\alpha}s+\int_b^t f(s)\Diamond_{\alpha}s,$
\item $\int_a^a f(s)\Diamond_{\alpha}s=0.$
\item If $f(t)\geq 0$ for all $t\in [a, b],$ then $\int_a^b f(t)\Diamond_{\alpha}t\geq 0.$
\item If $f(t)\leq g(t)$ for all $t\in [a, b],$ then $\int_a^b f(t)\Diamond_{\alpha}t\leq \int_a^b g(t)\Diamond_{\alpha}t.$
\end{enumerate}
\end{theorem}

The two-variable time scales calculus and multiple integration on time
scales were introduced in \cite{4,5} (see also \cite{6}). Let $\mathbb{T}%
_{1} $ and $\mathbb{T}_{2}$ be two time scales and put~{\small $\mathbb{T}%
_{1}\times \mathbb{T}_{2}=\left\{ \left( t,s\right) :t\in \mathbb{T}_{1},%
\text{ }s\in \mathbb{T}_{2}\right\} $}~which is a complete metric space with
the metric $d$ defined by{\small
\begin{equation*}
d\left( \left( t,s\right) ,\left( t^{\prime },s^{\prime }\right) \right) =%
\sqrt{\left( t-t^{\prime }\right) ^{2}+\left( s-s^{\prime }\right) ^{2}},%
\text{ \ \ }\forall \left( t,s\right) ,\left( t^{\prime },s^{\prime }\right)
\in \mathbb{T}_{1}\times \mathbb{T}_{2}.
\end{equation*}%
}

To make this paper self content, we now recall some basic definitions of the diamond-$\alpha$ partial dynamic calculus on time scales \cite{Ozkan}.\\

Suppose $a<b$ are points in $%
\mathbb{T}_{1},$ $c<d$ are points in $\mathbb{T}_{2},$ $[a,b]$ is the
closed bounded interval in $\mathbb{T}_{1}$, $[c,d]$ is the closed
bounded interval in $\mathbb{T}_{2}.$ We introduce a ``rectangle" in $%
\mathbb{T}_{1}\times $ $\mathbb{T}_{2}$ by$$R=[a, b]\times[c, d]=\{(t, s): t\in [a, b], ~s\in [c, d]\}.$$
Let $\{t_0,t_1,\ldots,t_n\}\subset[a,b]$, where $a=t_0<t_1<\cdots<t_n=b$ and $\{s_0,s_1,\ldots,s_k\}\subset[c,d],$ where $c=s_0<s_1<\cdots<s_k=d.$ The numbers $n$ and $k$ may be arbitrary positive integers. We call the collection of intervals $$P_1=\{[t_{i-1}, t_i):1\leq i\leq n\}$$ a $\Diamond_{\alpha}$-partition of $[a, b)$ and denote the set of all $\Diamond_{\alpha}$-partition of $[a, b)$ by $\mathcal{P}([a, b))$. Similarly, the collection of intervals $$P_2=\{[s_{j-1}, s_j):1\leq j\leq k\}$$ a $\Diamond_{\alpha}$-partition of $[c, d)$ and denote the set of all $\Diamond_{\alpha}$-partition of $[c, d)$ by $\mathcal{P}([c, d))$. Let $$R_{ij}=[t_{i-1}, t_i)\times [s_{j-1}, s_j), ~~{\rm where}~ 1\leq i\leq n, ~ 1\leq j\leq k.$$
We call the collection $P=\{R_{ij}: 1\leq i\leq n, ~1\leq j\leq k\}$ a $\Diamond_{\alpha}$-partition of $R,$ generated by the $\Diamond_{\alpha}$-partitions of $P_1$ and $P_2$ of $[a, b)$ and $[c, d)$ respectively, and write $P=P_1\times P_2$. The rectangles $P_{ij}, ~1\leq i\leq n, ~1\leq j\leq k$, are called the subrectangles of the partition $P$. The set of all $\Diamond_{\alpha}$-partitions of $R$ is denoted by $\mathcal{P}(R).$

Let $f: R\rightarrow\mathbb{R}$ be a bounded function. We set
$$M=\sup\{f(t, s): (t, s)\in R\}~~{\rm and} ~~m=\inf\{f(t, s): (t, s)\in R\}$$
and for $1\leq i\leq n,~1\leq j\leq k$, set $$M_{ij}=\sup\{f(t, s): (t, s)\in R_{ij}\}~~{\rm and} ~~m_{ij}=\inf\{f(t, s): (t, s)\in R_{ij}\}.$$
The upper Darboux $\Diamond_\alpha$-integral $U(f)$ of $f$ over $R$ and the lower Darboux $\Diamond_\alpha$-integral $L(f)$ of $f$ over $R$ are defined respectively by
$$U(f)=\inf\{U(f, P): P\in \mathcal{P}(R)\}$$
and
$$L(f)=\sup\{L(f, P): P\in \mathcal{P}(R)\},$$where
$$U(f, P)=\sum_{i=1}^n\sum_{j=1}^k M_{ij}(t_i - t_{i-1})(s_j - s_{j-1})$$
and
$$L(f, P)=\sum_{i=1}^n\sum_{j=1}^k m_{ij}(t_i - t_{i-1})(s_j - s_{j-1}).$$

\begin{definition}[\protect\cite{Ozkan}]
We say that $f$ is $\Diamond_\alpha$-integrable over $R$ provided $L(f)=U(f).$ In this case, we write $\int_R f(t, s)\Diamond_\alpha t\Diamond_\alpha s$ for this common value. We call this integral the Darboux $\Diamond_\alpha$-integral.
\end{definition}

\begin{theorem}[\protect\cite{Ozkan}]
If $L(f, P)=U(f, P)$ for some $\Diamond_\alpha$-partition of $R,$ then the function $f$ is $\Diamond_\alpha$-integrable over $R$ and $$\int_R f(t, s)\Diamond_\alpha t\Diamond_\alpha s=L(f, P)=U(f, P).$$
\end{theorem}

\section{Main Results}\label{sec:MR}
Let $t_\alpha:=\frac{1}{b-a}\int_a^b t \Diamond_{\alpha}t $ and $s_\alpha:=\frac{1}{d-c}\int_c^d s \Diamond_{\alpha}s. $
We now state and prove our first result.
\begin{theorem}\label{MR1}
Let $a,b,x\in {\mathbb{T}}_{1},~c,d,y\in {\mathbb{T}}_{2},$ with $a<b,~c<d$ and $f:[a,b]\times [c,d]\rightarrow \mathbb{R}$ be such that the partial mappings
\begin{equation*}
f_{y}:[a,b]\rightarrow \mathbb{R},~f_{y}(u):=f(u,y)~and~f_{x}:[c,d]
\rightarrow \mathbb{R},~f_{x}(v):=f(x,v)
\end{equation*}
defined for all $y\in [c,d]$ and $x\in [a,b],$ are continuous and convex. Then the following inequalities hold
\begin{align}\label{MRinq1}
\nonumber
&\frac{1}{2}\Bigg[\frac{1}{b-a}\int_a^bf(x, s_{\alpha})\Diamond_{\alpha}x+\frac{1}{d-c}\int_c^df(t_{\alpha},y)\Diamond_{\alpha}y \Bigg]\\
\nonumber
 &\leq \frac{1}{(b-a)(d-c)}\int_a^b\int_c^df(x, y)\Diamond_{\alpha}x\Diamond_{\alpha}y\\
\nonumber
&\leq \frac{1}{2(b-a)(d-c)}\int_a^b\Big[(d-s_{\alpha})f(x,c) +(s_{\alpha}-c)f(x,d)\Big] \Diamond_{\alpha}x\\
&+\frac{1}{2(b-a)(d-c)}\int_c^d\Big[(b-t_{\alpha})f(a, y)+(t_{\alpha}-a)f(b, y)\Big]\Diamond_{\alpha}y.
\end{align}
\end{theorem}

\begin{proof}
Applying Theorem \ref{lem1} to the function $f_y$, we obtain
\begin{align*}
f_y(t_{\alpha})\leq \frac{1}{b-a}\int_a^b f_y(x)\Diamond_{\alpha}x\leq \frac{b-t_{\alpha}}{b-a}f_y(a) + \frac{t_{\alpha}-a}{b-a}f_y(b),
\end{align*}
for all $y\in [c, d].$ That is,
\begin{align}\label{equ1}
f(t_{\alpha},y)\leq \frac{1}{b-a}\int_a^b f(x, y)\Diamond_{\alpha}x\leq \frac{b-t_{\alpha}}{b-a}f(a, y) + \frac{t_{\alpha}-a}{b-a}f(b, y).
\end{align}
Taking the diamond-$\alpha$ integral of both sides of (\ref{equ1}) over the interval $[c, d]$, applying items 1 and 7 of Theorem \ref{use}, and dividing by $d-c,$ we get
\begin{align}\label{equ2}
\nonumber
\frac{1}{d-c}\int_c^d&f(t_{\alpha},y)\Diamond_{\alpha}y\leq \frac{1}{(b-a)(d-c)}\int_a^b\int_c^df(x, y)\Diamond_{\alpha}x\Diamond_{\alpha}y\\
&\leq \frac{b-t_{\alpha}}{(b-a)(d-c)}\int_c^df(a, y)\Diamond_{\alpha}y + \frac{t_{\alpha}-a}{(b-a)(d-c)}\int_c^df(b, y)\Diamond_{\alpha}y.
\end{align}
Similarly, applying Theorem \ref{lem1} to $f_x,$ we get that for all $x\in[a, b],$

\begin{align}\label{equ22}
f(x, s_{\alpha})\leq \frac{1}{d-c}\int_c^d f(x, y)\Diamond_{\alpha}y\leq \frac{d-s_{\alpha}}{d-c}f(x, c) + \frac{s_{\alpha}-c}{d-c}f(x, d).
\end{align}
This implies that
\begin{align}\label{equ3}
\nonumber
\frac{1}{b-a}\int_a^b&f(x, s_{\alpha})\Diamond_{\alpha}x\leq \frac{1}{(b-a)(d-c)}\int_a^b\int_c^df(x, y)\Diamond_{\alpha}x\Diamond_{\alpha}y\\
&\leq \frac{d-s_{\alpha}}{(b-a)(d-c)}\int_a^bf(x,c)\Diamond_{\alpha}x + \frac{s_{\alpha}-c}{(b-a)(d-c)}\int_a^bf(x,d)\Diamond_{\alpha}x.
\end{align}
Adding relations (\ref{equ2}) and (\ref{equ3}), we obtain
\begin{align*}
&\frac{1}{d-c}\int_c^df(t_{\alpha},y)\Diamond_{\alpha}y +\frac{1}{b-a}\int_a^bf(x, s_{\alpha})\Diamond_{\alpha}x\\
 &\leq \frac{2}{(b-a)(d-c)}\int_a^b\int_c^df(x, y)\Diamond_{\alpha}x\Diamond_{\alpha}y\\
&\leq \frac{b-t_{\alpha}}{(b-a)(d-c)}\int_c^df(a, y)\Diamond_{\alpha}y + \frac{t_{\alpha}-a}{(b-a)(d-c)}\int_c^df(b, y)\Diamond_{\alpha}y\\
&+\frac{d-s_{\alpha}}{(b-a)(d-c)}\int_a^bf(x,c)\Diamond_{\alpha}x + \frac{s_{\alpha}-c}{(b-a)(d-c)}\int_a^bf(x,d)\Diamond_{\alpha}x.
\end{align*}
Hence (\ref{MRinq1}) follows.
\end{proof}

\begin{corollary}\label{cor11}
For $\alpha=0,$ the inequalities in Theorem \ref{MR1} become
\begin{align}
\nonumber
&\frac{1}{2}\Bigg[\frac{1}{b-a}\int_a^bf(x, s_0)\nabla x+\frac{1}{d-c}\int_c^df(t_0,y)\nabla y \Bigg]\\
\nonumber
 &\leq \frac{1}{(b-a)(d-c)}\int_a^b\int_c^df(x, y)\nabla x\nabla y\\
\nonumber
&\leq \frac{1}{2(b-a)(d-c)}\int_a^b\Big[(d-s_0)f(x,c) +(s_0-c)f(x,d)\Big] \nabla x\\
&+\frac{1}{2(b-a)(d-c)}\int_c^d\Big[(b-t_0)f(a, y)+(t_0-a)f(b, y)\Big]\nabla y,
\end{align}
where $s_0=\frac{1}{d-c}\int_c^d s\nabla s$ and $t_0=\frac{1}{b-a}\int_a^b t\nabla t.$
\end{corollary}

\begin{corollary}\label{cor12}
For $\alpha=1/2,$ the inequalities in Theorem \ref{MR1} become
\begin{align}
\nonumber
&\frac{1}{2}\Bigg[\frac{1}{b-a}\int_a^bf\Big(x, \frac{c+d}{2}\Big)\Diamond_{1/2}x+\frac{1}{d-c}\int_c^df\Big(\frac{a+b}{2},y\Big)\Diamond_{1/2}y \Bigg]\\
\nonumber
 &\leq \frac{1}{(b-a)(d-c)}\int_a^b\int_c^df(x, y)\Diamond_{1/2}x\Diamond_{1/2}y\\
\nonumber
&\leq \frac{1}{4(b-a)}\int_a^b\Big[f(x,c) +f(x,d)\Big] \Diamond_{1/2}x\\
&+\frac{1}{4(d-c)}\int_c^d\Big[f(a, y)+f(b, y)\Big]\Diamond_{1/2}y.
\end{align}
\end{corollary}

\begin{proof}
From (3.31) in \cite{Dinu}, we have that
\begin{equation}\label{u}
\int_a^b t \Diamond_{1/2}t=\frac{b^2-a^2}{2}.
\end{equation}
 The desired inequality follows by applying (\ref{u}) to the definitions of $t_{\alpha}$ and $s_{\alpha}$ when $\alpha=1/2.$
\end{proof}

\begin{corollary}\label{cor13}
For $\alpha=1,$ the inequalities in Theorem \ref{MR1} become
\begin{align}
\nonumber
&\frac{1}{2}\Bigg[\frac{1}{b-a}\int_a^bf(x, s_1)\Delta x+\frac{1}{d-c}\int_c^df(t_1,y)\Delta y \Bigg]\\
\nonumber
 &\leq \frac{1}{(b-a)(d-c)}\int_a^b\int_c^df(x, y)\Delta x\Delta y\\
\nonumber
&\leq \frac{1}{2(b-a)(d-c)}\int_a^b\Big[(d-s_1)f(x,c) +(s_1-c)f(x,d)\Big] \Delta x\\
&+\frac{1}{2(b-a)(d-c)}\int_c^d\Big[(b-t_1)f(a, y)+(t_1-a)f(b, y)\Big]\Delta y,
\end{align}
where $s_1=\frac{1}{d-c}\int_c^d s\Delta s$ and $t_1=\frac{1}{b-a}\int_a^b t\Delta t.$
\end{corollary}

\begin{theorem}\label{MR2}
Under the assumption of Theorem \ref{MR1}, and suppose also the intervals contain the mid points, then we have the following inequality
\begin{align}\label{MRineq2}
\nonumber
f\Big(\frac{a+b}{2}, s_{\alpha}\Big)+f\Big(t_{\alpha},\frac{c+d}{2}\Big)\leq \frac{1}{b-a}\int_a^b &f\Big(x, \frac{c+d}{2}\Big)\Diamond_{\alpha}x\\
&+ \frac{1}{d-c}\int_c^d f\Big(\frac{a+b}{2}, y\Big)\Diamond_{\alpha}y.
\end{align}
\end{theorem}
\begin{proof}
Since the inequalities in (\ref{equ1}) hold for all $y\in[c, d]$, and by using the assumption of our theorem,  we therefore have that for $y=\frac{c+d}{2}$, the following inequality holds
\begin{align}\label{equ4}
f\Big(t_{\alpha},\frac{c+d}{2}\Big)\leq \frac{1}{b-a}\int_a^b f\Big(x, \frac{c+d}{2}\Big)\Diamond_{\alpha}x.
\end{align}
Using a similar argument, we get from (\ref{equ22})
\begin{align}\label{equ5}
f\Big(\frac{a+b}{2}, s_{\alpha}\Big)\leq \frac{1}{d-c}\int_c^d f\Big(\frac{a+b}{2}, y\Big)\Diamond_{\alpha}y.
\end{align}
Inequality (\ref{MRineq2}) follows by adding (\ref{equ4}) and (\ref{equ5}).
\end{proof}

\begin{corollary}\label{cor21}
If we take $\alpha=0,$ then the inequality in Theorem \ref{MR2} reduces to
\begin{align}\label{ineCor21}
\nonumber
f\Big(\frac{a+b}{2}, s_0\Big)+f\Big(t_0,\frac{c+d}{2}\Big)\leq \frac{1}{b-a}\int_a^b &f\Big(x, \frac{c+d}{2}\Big)\nabla x\\
&+ \frac{1}{d-c}\int_c^d f\Big(\frac{a+b}{2}, y\Big)\nabla y,
\end{align}
where $s_0=\frac{1}{d-c}\int_c^d s\nabla s$ and $t_0=\frac{1}{b-a}\int_a^b t\nabla t.$
\end{corollary}

\begin{corollary}\label{cor22}
For $\alpha=1/2,$ Theorem \ref{MR2} becomes
\begin{align}\label{ineCor22}
\nonumber
f\Big(\frac{a+b}{2},\frac{c+d}{2}\Big)\leq \frac{1}{2(b-a)}\int_a^b &f\Big(x, \frac{c+d}{2}\Big)\Diamond_{1/2} x\\
&+ \frac{1}{2(d-c)}\int_c^d f\Big(\frac{a+b}{2}, y\Big)\Diamond_{1/2} y.
\end{align}
\end{corollary}

\begin{corollary}\label{cor23}
Setting $\alpha=1$ in Theorem \ref{MR2}, we get
\begin{align}\label{ineCor23}
\nonumber
f\Big(\frac{a+b}{2}, s_1\Big)+f\Big(t_1,\frac{c+d}{2}\Big)\leq \frac{1}{b-a}\int_a^b &f\Big(x, \frac{c+d}{2}\Big)\Delta x\\
&+ \frac{1}{d-c}\int_c^d f\Big(\frac{a+b}{2}, y\Big)\Delta y,
\end{align}
where $s_1=\frac{1}{d-c}\int_c^d s\Delta s$ and $t_1=\frac{1}{b-a}\int_a^b t\Delta t.$
\end{corollary}

\begin{theorem}\label{MR3}
Under the assumption of Theorem \ref{MR1}, we have the following inequality
\begin{align}\label{LI}
\nonumber
&\frac{1}{b-a}\int_a^b \big[f(x, c)+ f(x, d)\big]\Diamond_{\alpha}x + \frac{1}{d-c}\int_c^d \big[f(a, y)+f(b,y)\big]\Diamond_{\alpha}y\\
&\leq A_1f(a, c) + A_2f(a, d) + A_3f(b, c) + A_4f(b, d),
\end{align}
where $A_1=\frac{b-t_{\alpha}}{b-a}+\frac{d-s_{\alpha}}{d-c},$ $A_2=\frac{b-t_{\alpha}}{b-a}+\frac{s_{\alpha}-c}{d-c},$  $A_3=\frac{t_{\alpha}-a}{b-a}+\frac{d-s_{\alpha}}{d-c},$  and $A_4=\frac{t_{\alpha}-a}{b-a}+\frac{s_{\alpha}-c}{d-c}.$
\end{theorem}

\begin{proof}
From (\ref{equ1}) and (\ref{equ22}), we get the following inequalities
\begin{align*}
\frac{1}{b-a}\int_a^b f(x, c)\Diamond_{\alpha}x\leq \frac{b-t_{\alpha}}{b-a}f(a, c) + \frac{t_{\alpha}-a}{b-a}f(b, c),
\end{align*}
\begin{align*}
\frac{1}{b-a}\int_a^b f(x, d)\Diamond_{\alpha}x\leq \frac{b-t_{\alpha}}{b-a}f(a, d) + \frac{t_{\alpha}-a}{b-a}f(b, d),
\end{align*}
\begin{align*}
\frac{1}{d-c}\int_c^d f(a, y)\Diamond_{\alpha}y\leq \frac{d-s_{\alpha}}{d-c}f(a, c) + \frac{s_{\alpha}-c}{d-c}f(a, d),
\end{align*}
and
\begin{align*}
\frac{1}{d-c}\int_c^d f(b, y)\Diamond_{\alpha}y\leq \frac{d-s_{\alpha}}{d-c}f(b, c) + \frac{s_{\alpha}-c}{d-c}f(b, d),
\end{align*}
which give, by addition, (\ref{LI}).

\end{proof}

\begin{corollary}\label{corLI1}
For $\alpha=1,$ we get from Theorem \ref{MR3} the following inequality
\begin{align}
\nonumber
&\frac{1}{b-a}\int_a^b \big[f(x, c)+ f(x, d)\big]\Delta x + \frac{1}{d-c}\int_c^d \big[f(a, y)+f(b,y)\big]\Delta y\\
&\leq A_1f(a, c) + A_2f(a, d) + A_3f(b, c) + A_4f(b, d),
\end{align}
where $A_1=\frac{b-t_1}{b-a}+\frac{d-s_1}{d-c},$ $A_2=\frac{b-t_1}{b-a}+\frac{s_1-c}{d-c},$  $A_3=\frac{t_1-a}{b-a}+\frac{d-s_1}{d-c},$  and $A_4=\frac{t_1-a}{b-a}+\frac{s_1-c}{d-c}.$
\end{corollary}
\begin{remark}
If we take ${\mathbb{T}}_{1}={\mathbb{T}}_{2}=\mathbb{R}$ in Corollaries \ref{cor13}, \ref{cor23} and \ref{corLI1}, and combine the resultant inequalities, we get Theorem \ref{drag} due to Dragomir \cite{Drag}.
\end{remark}

As an application of the above corollary, we consider the following example.

\begin{example}
Let ${\mathbb{T}}_{1}={\mathbb{T}}_{2}=\mathbb{Z}$, and $f: [0, 2]\times[1, 3]\rightarrow\mathbb{R}$ be convex and continuous on the coordinates. Then the following inequality holds
$$f(0, 2)+f(1, 1)+f(1, 3)+ f(2, 2)\leq f(0, 1)+ f(0, 3)+ f(2, 1)+ f(2, 3).$$
\end{example}
\noindent{Justification}:\\
The left hand side of Corollary \ref{corLI1} gives
\begin{align}\label{e1}
\nonumber
&\frac{1}{2}\int_0^2 \big[f(x, 1)+ f(x, 3)\big]\Delta x + \frac{1}{2}\int_1^3 \big[f(0, y)+f(2,y)\big]\Delta y\\
\nonumber
&= \frac{1}{2}\sum_{x=0}^{1}\big[f(x, 1)+ f(x, 3)\big] + \frac{1}{2}\sum_{y=1}^{2}\big[f(0, y)+f(2,y)\big]\\
&= f(0, 1)+ \frac{1}{2}\Big[f(0,2)+ f(0, 3)+ f(1, 1)+ f(1, 3) + f(2, 1) + f(2, 2)\Big].
\end{align}
By a simple computation we get $t_1=\frac{1}{2}$, $s_1=\frac{3}{2}$, $$A_1=\frac{3}{2}, ~A_2=1=A_3, ~{\rm and}~A_4=\frac{1}{2}.$$
Hence, the right hand side of  Corollary \ref{corLI1}  amounts to
\begin{equation}\label{e2}
\frac{3}{2}f(0, 1)+ f(0, 3)+ f(2, 1)+ \frac{1}{2}f(2, 3).
\end{equation}
We obtain the desired inequality by combining (\ref{e1}) and (\ref{e2}) in the spirit of Corollary \ref{corLI1}.

\section{Conclusion}
A time scale version of the Hermite--Hadamard inequality for functions convex on the coordinates has been proved. By taking ${\mathbb{T}}_{1}={\mathbb{T}}_{2}=\mathbb{R}$ in Corollaries \ref{cor13}, \ref{cor23} and \ref{corLI1}, and combining the resultant inequalities, we get the result of Dragomir\cite{Drag}. Since in some cases the diamond-$\alpha$ derivative is a particular case of the symmetric derivative on time scales, it would be interesting to see, in further work, if our results are also valid for symmetric calculus on time scales. For more on the symmetric calculus, see \cite{Brito1,Brito2}.
\section*{Acknowledgements}
Many thanks to the three anonymous referees for meticulously checking the details and providing helpful comments that improved this paper.


\bibliographystyle{plain}

\label{lastpage}
\end{document}